\documentclass[10pt]{amsart}

\usepackage[latin1]{inputenc}
\usepackage{amsmath}
\usepackage{amsfonts}
\usepackage{amssymb}
\usepackage[mathscr]{eucal}
\usepackage{diagrams}
\usepackage{xy,color,graphicx}
\usepackage{hyperref}
\xyoption{all}

\newcommand{\wis}[1]{{\text{\em \usefont{OT1}{cmtt}{m}{n} #1}}}

\newcommand{\C}{\mathbb{C}}
\newcommand{\Z}{\mathbb{Z}}
\newcommand{\vtx}[1]{*+[o][F-]{\scriptscriptstyle #1}}
\newcommand{\Oscr}{\mathcal{O}}

\newtheorem{theorem}{Theorem}
\newtheorem{lemma}{Lemma}

\title{Most irreducible representations of the $3$-string braid group}
\author{Lieven Le Bruyn} 
\address{Department of Mathematics, University of Antwerp \\ 
 Middelheimlaan 1, B-2020 Antwerp (Belgium) \\ {\tt lieven.lebruyn@ua.ac.be}}

\begin{document}
\sloppy

\maketitle

\section{Introduction}

With $\wis{iss}_n~B_3$ we denote the affine variety of all isomorphism classes of semi-simple $n$-dimensional representations of the 3-string braid group
\[
B_3 = \langle \sigma_1,\sigma_2~|~\sigma_1 \sigma_2 \sigma_1 = \sigma_2 \sigma_1 \sigma_2 \rangle \]
It is well-known, see for example \cite{Westbury}, \cite{LBbraid1} and \cite{LBbraid2}, that any irreducible components $X_{\sigma}$ of $\wis{iss}_n~B_3$ containing a Zariski open subset of irreducible representations is determined by a dimension-vector
$\sigma=(a,b;x,y,z)$ satisfying
\[
n=a+b=x+y+z \quad \text{and} \quad x=\wis{max}(x,y,z) \leq b=\wis{min}(a,b) \]
with $dim~X_{\sigma} = n_{\sigma} = 1+n^2-(a^2+b^2+x^2+y^2+z^2)$. As $B_3$ is of wild representation type one cannot expect a full classification of all its finite dimensional irreducible representations. In fact, such a classification is only known for $n \leq 5$ by work of Imre Tuba and Hans Wenzl \cite{TubaWenzl}. Still, one can aim to describe 'most' irreducible representations by constructing for each component $X_{\sigma}$ an explicit minimal (\'etale) rational map
\[
f_{\sigma}~:~\mathbb{A}^{n_{\sigma}} \rDotsto X_{\sigma} \rInto \wis{iss}_n~B_3 \]
having a Zariski dense image. Such rational dense parametrizations were constructed in \cite{LBbraid1} for all components when $n < 12$. The purpose of the present paper is to extend this to all finite dimensions $n$. 

\section{Linear systems and some rational quiver settings}

A linear control system $\Sigma$ is determined by the system of linear differential equations
\[
\begin{cases}
\dot{x} = Ax+Bu \\
y = Cx
\end{cases}
\]
where $\Sigma=(A,B,C) \in M_n(\C) \times M_{n \times m}(\C) \times M_{p \times n}(\C)$ and $u(t) \in \C^m$ is the control at time $t$, $x(t) \in \C^n$ is the state of the system and $y(t) \in \C^p$ its output. Equivalent control systems differ only by a base change in the state space, that is $\Sigma'=(A',B',C')$ is equivalent to $\Sigma$ if and only if there exists a $g \in GL_n(\C)$ such that
\[
A'=gAg^{-1}, \qquad B'=gB \quad \text{and} \quad C' = Cg^{-1} \]
$\Sigma$ is said to be {\em canonical} if the matrices
\[
c_{\Sigma} =\begin{bmatrix} B & AB & A^2B & \hdots & A^{n-1}B \end{bmatrix} \quad \text{and} \quad o_{\Sigma} = \begin{bmatrix} C & CA & CA^2 & \hdots & CA^{n-1} \end{bmatrix} \]
are of maximal rank. 

Michiel Hazewinkel proved in \cite{Hazewinkel} that the moduli space $\wis{sys}^c_{m,n,p}$ of all such canonical linear systems is a smooth rational quasi-affine variety of dimension $(m+p)n$. We will give another short proof of this result and draw some consequences from it (see also \cite{LBReineke}).

Consider the quiver setting with $m$ arrows $\{ b_1,\hdots,b_m \}$ from left to right and $p$ arrows $\{ c_1,\hdots,c_p \}$ from right to left
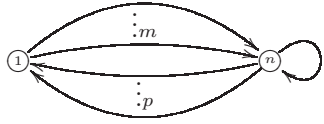
\[
\xymatrix@=3cm{
\vtx{1} \ar@/^/[r]^{\vdots m} \ar@/^5ex/[r] & \vtx{n} \ar@(ur,dr) \ar@/^/[l] \ar@/^5ex/[l]_{\vdots p} }
\]
To a system $\Sigma=(A,B,C)$ we associate the quiver-representation $V_{\Sigma}$ by assigning to the arrow $b_i$ the $i$-th column $B_i$ of the matrix $B$, to the arrow $c_j$ the $j$-th row $C^j$ of $C$ and the matrix $A$ to the loop. As the base change group $\C^* \times GL_n$ acts on these quiver-representations by
\[
(\lambda,g).V_{\Sigma} = (gAg^{-1},gB_1 \lambda^{-1},\hdots,g B_m \lambda^{-1}, \lambda C^1 g^{-1}, \hdots, \lambda C^p g^{-1}) \]
with the subgroup $\C^*(1,1_n)$ acting trivially, there is a natural one-to-one correspondence between equivalence classes of linear systems $\Sigma$ and isomorphism classes of quiver-representations $V_{\Sigma}$. Under this correspondence it is easy to see that canonical systems correspond to {\em simple} quiver-representations, see \cite[Lemma 1]{LBReineke}. Hence, the moduli-space $\wis{sys}^c_{m,n,p}$ is isomorphic to the Zariski-open subset of the affine quotient-variety classifying isomorphism classes of semi-simple quiver-representations, proving smoothness, quasi-affineness as well as determining the dimension by general results, see for example \cite{LBBook}.

\begin{lemma} \label{lemma1} A generic canonical system $\Sigma$ is equivalent to a triple $(A_n,B^{\bullet}_{nm},C_{pn})$ with
\[
A_n = \begin{bmatrix} 0 & 0 & \hdots & & x_n \\
1 & 0 & \hdots & & x_{n-1} \\
& \ddots & \ddots & & \vdots \\
& & \ddots & 0 & x_2 \\
& & & 1 & x_1 \end{bmatrix} \qquad
B^{\bullet}_{nm} = \begin{bmatrix} 1 & b_{12} & \hdots & b_{1m} \\ 0 & b_{22} & \hdots & b_{2m} \\ \vdots & \vdots & & \vdots \\ 0 & b_{n2} & \hdots & b_{nm} \end{bmatrix} \]
that is, where $A_n$ is a companion $n \times n$-matrix, $B^{\bullet}_{nm}$ is the generic $n \times m$-matrix with fixed first column and $C_{pn}$ a generic $p \times n$-matrix.
\end{lemma}

\begin{proof} A generic representation of the quiver-setting
\[
\xymatrix{\vtx{1} \ar[rr]^v & & \vtx{n} \ar@(ur,dr)^{A}} \]
will have the property that $v$ is a cyclic-vector for the matrix $A$, that is, $\{ v, Av, A^2v, \hdots, A^{n-1}v \}$ are linearly independent. But then, performing a base-change we get a representation of the form
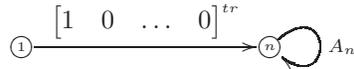
\[
\xymatrix{\vtx{1} \ar[rrr]^{\begin{bmatrix} 1 & 0 & \hdots & 0 \end{bmatrix}^{tr}} & & & \vtx{n} \ar@(ur,dr)^{A_n}} \]
where $A_n$ is a companion matrix whose $n$-th column expresses the vector $-A^n v$ in the new basis. As the automorphism group of this representation is reduced to $\C^*(1,1_n)$, any general representation $V_{\Sigma}$ is isomorphic to one with $B_1 = \begin{bmatrix} 1 & 0 & \hdots & 0 \end{bmatrix}^{tr}$, $A=A_n$ and the other columns of $B$ and all rows of $C$ generic vectors.
\end{proof}

\begin{lemma} \label{lemma2} The following representations give a rational parametrization of the isomorphism classes of simple representations of these quiver-settings
\[
R_k~\qquad~:~\qquad~\xymatrix{\vtx{1} \ar@/^/[rrr]^{\begin{bmatrix} 1 & 0 & \hdots & 0 \end{bmatrix}^{tr}} & & & \vtx{k} \ar@/^/[rr]^{A_k} \ar@/^/[lll]^{\begin{bmatrix} y_1 & y_2 & \hdots & y_k \end{bmatrix}} & & \vtx{k} \ar@/^/[ll]^{1_k}} \]
and
\[
S_k~\qquad~:~\qquad~\xymatrix{\vtx{1} \ar@/^/[rrr]^{\begin{bmatrix} 1 & 0 & \hdots & 0 \end{bmatrix}^{tr}} & & & \vtx{k} \ar@/^/[rr]^{A^{\dagger}_k} \ar@/^/[lll]^{\begin{bmatrix} y_1 & y_2 & \hdots & y_k \end{bmatrix}} & & \vtx{k-1} \ar@/^/[ll]^{\begin{bmatrix} \underline{0} \\ 1_{k-1} \end{bmatrix}}} \]
where $A_k$ (reps. $A_k^{\dagger}$) is the generic $k \times k$ companion matrix (resp. the reduced $k-1 \times k$ companion matrix)
\[
A_k = \begin{bmatrix} 0 & 0 & \hdots & & x_k \\
1 & 0 & \hdots & & x_{k-1} \\
& \ddots & \ddots & & \vdots \\
& & \ddots & 0 & x_2 \\
& & & 1 & x_1 \end{bmatrix}~\quad~\text{and}~\quad~A_k^{\dagger} = \begin{bmatrix} 
1 & 0 & \hdots & & x_{k-1} \\
& \ddots & \ddots & & \vdots \\
& & \ddots & 0 & x_2 \\
& & & 1 & x_1 \end{bmatrix} \]
\end{lemma}

\begin{proof} By invoking the first fundamental theorem of $GL_n$-invariants (see for example \cite[Thm. II.4.1]{Kraft}) we can in case $R_k$ eliminate the base-change action in the right-most vertex, giving a natural one-to-one correspondence between isoclasses of representations
\[
\xymatrix{\vtx{1} \ar@/^/[rr]^v & & \vtx{k} \ar@/^/[rr]^X \ar@/^/[ll]^{w^{\tau}} & & \vtx{k} \ar@/^/[ll]^Y} \quad \leftrightarrow \quad \xymatrix{\vtx{1}  \ar@/^/[rr]^v & & \vtx{k} \ar@/^/[ll]^{w^{\tau}}  \ar@(ur,dr)^{Y.X}} \]
and hence the claim follows from the previous lemma. As for case $S_k$ we can again apply the first fundamental theorem for $GL_n$-invariants, now with respect to the base-change action in the middle vertex, to obtain a natural one-to-one correspondence between isoclasses of representations
\[
\xymatrix{\vtx{1} \ar@/^/[rr]^v & & \vtx{k} \ar@/^/[rr]^X \ar@/^/[ll]^{w^{\tau}} & & \vtx{k-1} \ar@/^/[ll]^Y} \quad \leftrightarrow \quad \xymatrix{\vtx{1} \ar@(ul,dl)_{w^{\tau}.v} \ar@/^/[rr]^{X.v} & & \vtx{k-1} \ar@(ur,dr)^{X.Y} \ar@/^/[ll]^{w^{\tau}.Y}} \]
and again the claim follows from the previous lemma, taking into account the extra free loop in the left-most vertex, which corresponds to $y_1$.
\end{proof}

\begin{lemma} \label{case1}
The following representations give a rational parametrization for the isomorphism classes of simple representations of the quiver-setting
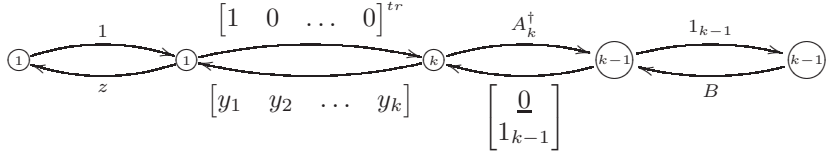
\[
\xymatrix{\vtx{1} \ar@/^/[rr]^{1} & & \vtx{1} \ar@/^/[ll]^{z} \ar@/^/[rrr]^{\begin{bmatrix} 1 & 0 & \hdots & 0 \end{bmatrix}^{tr}} & & & \vtx{k} \ar@/^/[rr]^{A^{\dagger}_k} \ar@/^/[lll]^{\begin{bmatrix} y_1 & y_2 & \hdots & y_k \end{bmatrix}} & & \vtx{k-1} \ar@/^/[ll]^{\begin{bmatrix} \underline{0} \\ 1_{k-1} \end{bmatrix}} \ar@/^/[rr]^{1_{k-1}} & & \vtx{k-1} \ar@/^/[ll]^{B}} 
\]
where $B$ is a generic $k-1 \times k-1$ matrix and, as before, $A_k^{\dagger}$ is a reduced generic companion matrix.
\end{lemma}

\begin{proof} 
Forgetting the end-vertices (and maps to and from them) we are in the situation of the previous lemma. For general values these are simple quiver-representations and hence the automorphism group is reduced to $\C^*(1,1_k,1_{k-1})$. If we now add the end vertices we can use base-change in them to force one of the two arrows to be the identity map, leaving the remaining map generic. Alternatively, we can use the first fundamental theorem of $GL_n$-invariants as before, to obtain the claimed result.
\end{proof}

\section{Luna slices and the action map}

We quickly recall the basic strategy of \cite{LBbraid1}. As the central generator $c=(\sigma_1 \sigma_2)^3 = (\sigma_1 \sigma_2 \sigma_1)^2$ of $B_3$ acts via a scalar $\lambda \in \C^*$ on any irreducible $B_3$-representation it suffices to study irreducible representations of the quotient group $B_3 / \langle c \rangle \simeq C_2 \ast C_3 = \langle s,t~|~s^2=e=t^3 \rangle$ where $s$ is the class of $\sigma_1 \sigma_2 \sigma_1$ and $t$ that of $\sigma_1 \sigma_2$. Note that this quotient-group is isomorphic to the modular group $PSL_2(\Z)$. The action of $s$ and $t$ on a finite dimensional $C_2 \ast C_3$-representation $V$ induce two decompositions of $V$ into eigen-spaces
\[
V_+ \oplus V_- = V = V_1 \oplus V_{\rho} \oplus V_{\rho^2} \]
where $\rho$ is a primitive $3$-rd root of unity. Hence $V$ is fully determined by a base-change matrix $B=(B_{ij})_{1 \leq i \leq 3, 1 \leq j \leq 2}$ from a fixed basis compatible with the first decomposition to a fixed basis compatible with the second, that is by a representation of the quiver-setting
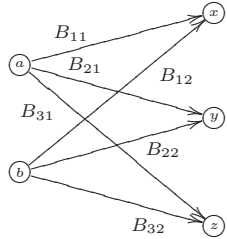
\[
\xymatrix@=.4cm{
& & & & \vtx{x} \\
\vtx{a} \ar[rrrru]^(.3){B_{11}} \ar[rrrrd]^(.3){B_{21}} \ar[rrrrddd]_(.2){B_{31}} & & & & \\
& & & & \vtx{y} \\
\vtx{b} \ar[rrrruuu]_(.7){B_{12}} \ar[rrrru]_(.7){B_{22}} \ar[rrrrd]_(.7){B_{32}} & & & & \\
& & & & \vtx{z}}
\]
Bruce Westbury observed in \cite{Westbury} that under this correspondence isoclasses of $C_2 \ast C_3$-representations coincide with isoclasses of quiver-representations, and that irreducible group-representations correspond to stable quiver-representations wrt. the stability structure $\theta=(-1,-1;1,1,1)$. It then follows from this stability condition that the dimension-vectors $\sigma = (a,b;x,y,z)$ containing a Zariski open subset of irreducible $n$-dimensional $C_2 \ast C_3$-representations must satisfy $a+b=n=x+y+z$ as well as $max(x,y,z) \leq min(a,b)$. 

Working backwards, we obtain for each $\lambda \in \C^*$ an irreducible $B_3$-representation determined by the above base-change matrix $B$ via
\[
( \ast )  \begin{cases}
\sigma_1 \mapsto  \lambda^{1/6} B^{-1} \begin{bmatrix} 1_{x} & 0 & 0 \\ 0 & \rho^2 1_{y} & 0 \\ 0 & 0 & \rho 1_{z} \end{bmatrix} B \begin{bmatrix} 1_{a} & 0 \\ 0 & -1_{b} \end{bmatrix} \\
\sigma_2 \mapsto  \lambda^{1/6} \begin{bmatrix} 1_{a} & 0 \\ 0 & -1_{b} \end{bmatrix} B^{-1}  \begin{bmatrix} 1_{x} & 0 & 0 \\ 0 & \rho^2 1_{y} & 0 \\ 0 & 0 & \rho 1_{z} \end{bmatrix} B
\end{cases}
\]
Observe that in lifting irreducibles from $C_2 \ast C_3$ to $B_3$ we get an action by multiplication of $6$-th roots of unity on the components which contain irreducibles, which accounts for the fact that the irreducible  components $X_{\sigma}$ containing irreducible $B_3$-representations are classified by the dimension vectors $\sigma=(a,b;x,y,z)$ as above with the extra condition that $b=min(a,b)$ and $x=max(x,y,z)$. We will now construct  special semi-simple $C_2 \ast C_3$-representations $M_0$ in every component, with all its irreducible factors being $1$- or $2$-dimensional.

There are $6$ one-dimensional irreducible $C_2 \ast C_3$-representations, coresponding to the quiver-representations $S_i$ for $1 \leq i \leq 6$:
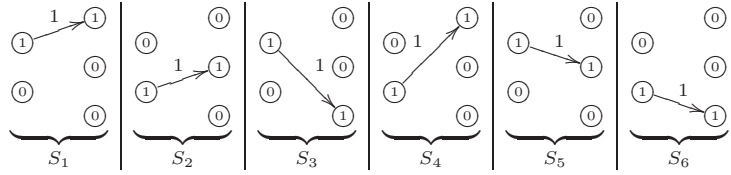
\[
\begin{array}{c|c|c|c|c|c}
\underbrace{\xymatrix@=.001cm{
& & & & \vtx{1} \\
\vtx{1} \ar[rrrru]^1  & & & & \\
& & & & \vtx{0} \\
\vtx{0}  & & & & \\
& & & & \vtx{0}}}_{S_1} &
\underbrace{\xymatrix@=.001cm{
& & & & \vtx{0} \\
\vtx{0}  & & & & \\
& & & & \vtx{1} \\
\vtx{1}  \ar[rrrru]^1  & & & & \\
& & & & \vtx{0}}}_{S_2} &
\underbrace{\xymatrix@=.001cm{
& & & & \vtx{0} \\
\vtx{1}  \ar[rrrrddd]^1 & & & & \\
& & & & \vtx{0} \\
\vtx{0}  & & & & \\
& & & & \vtx{1}} }_{S_3}
&
\underbrace{\xymatrix@=.001cm{
& & & & \vtx{1} \\
\vtx{0}  & & & & \\
& & & & \vtx{0} \\
\vtx{1} \ar[rrrruuu]^1  & & & & \\
& & & & \vtx{0}}}_{S_4} &
\underbrace{\xymatrix@=.001cm{
& & & & \vtx{0} \\
\vtx{1}  \ar[rrrrd]^1  & & & & \\
& & & & \vtx{1} \\
\vtx{0}  & & & & \\
& & & & \vtx{0}}}_{S_5} &
\underbrace{\xymatrix@=.001cm{
& & & & \vtx{0} \\
\vtx{0}  & & & & \\
& & & & \vtx{0} \\
\vtx{1}  \ar[rrrrd]^1 & & & & \\
& & & & \vtx{1}} }_{S_6}
\end{array}
\]
and three one-parameter families of two-dimensional irreducibles corresponding to the quiver-representations $T_i(q)$ for $q \not= 0,1$ and $1 \leq i \leq 3$
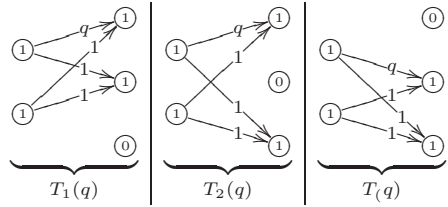
\[
\begin{array}{c|c|c}
\underbrace{\xymatrix@=.1cm{
& & & & \vtx{1} \\
\vtx{1} \ar[rrrru]|(0.6){q} \ar[rrrrd]|(0.6){1}  & & & & \\
& & & & \vtx{1} \\
\vtx{1} \ar[rrrru]|(0.6){1} \ar[rrrruuu]|(0.7){1}  & & & & \\
& & & &  \vtx{0} }}_{T_1(q)}
&
\underbrace{\xymatrix@=.1cm{
& & & & \vtx{1} \\
\vtx{1} \ar[rrrru]|(0.6){q}  \ar[rrrrddd]|(0.6){1} & & & & \\
& & & & \vtx{0} \\
\vtx{1}  \ar[rrrruuu]|(0.6){1} \ar[rrrrd]|(0.6){1} & & & & \\
& & & &  \vtx{1} }}_{T_2(q)} &
\underbrace{\xymatrix@=.1cm{
& & & & \vtx{0} \\
\vtx{1}  \ar[rrrrd]|(0.6){q} \ar[rrrrddd]|(0.7){1} & & & & \\
& & & & \vtx{1} \\
\vtx{1} \ar[rrrru]|(0.6){1}  \ar[rrrrd]|(0.6){1} & & & & \\
& & & &  \vtx{1} }}_{T_(q)}
\end{array}
\]
The semi-simple representation
\[
M_0 = S_1^{\oplus a_1} \oplus S_2^{\oplus a_2} \oplus S_3^{\oplus a_3} \oplus S_4^{\oplus a_4} \oplus S_5^{\oplus a_5} \oplus S_6^{\oplus a_6} \oplus T_1(q)^{\oplus b_{\alpha}} \oplus T_2(q)^{\oplus b_{\beta}} \oplus T_3(q)^{\oplus b_{\gamma}} \]
clearly belongs to the component $X_{\sigma}$ with dimension vector $\sigma=(a,b;x,y,z)$ where
\[
\begin{cases}
a=a_1+a_3+a_5+b_{\alpha} + b_{\beta} \\
b=a_2+a_4+a_6+b_{\alpha}+b_{\gamma} \\
x=a_1+a_4+b_{\alpha}+b_{\beta} \\
y=a_2+a_5+b_{\alpha}+b_{\gamma} \\
z=a_3+a_6+b_{\beta}+b_{\gamma}
\end{cases}
\]
and is fully determined by the base-change matrix $B_0$ with block-form as above
\[
\begin{array}{|cccccc|cccccc|}
1_{a_1} & 0 & 0 & 0 & 0 & 0 & 0 & 0 & 0 & 0 & 0 & 0 \\
0 & 0 & 0 & 0 & 0 & 0 & 0 & 1_{a_4} & 0 & 0 & 0 & 0 \\
0 & 0 & 0 & q 1_{b_{\alpha}}  & 0 & 0 & 0 & 0 & 0 & 1_{b_{\alpha}} & 0 & 0 \\
0 & 0 & 0 & 0 & q 1_{b_{\beta}} & 0 & 0 & 0 & 0 & 0 & 1_{b_{\beta}} & 0  \\
\hline
0 & 0 & 0 & 0 & 0 & 0 & 1_{a_2} & 0 & 0 & 0 & 0 & 0 \\
0 & 0 & 1_{a_5} & 0 & 0 & 0 & 0 & 0 & 0 & 0 & 0 & 0 \\
0 & 0 & 0 & 1_{b_{\alpha}} & 0 & 0 & 0 & 0 & 0 & 1_{b_{\alpha}} & 0 & 0 \\
0 & 0 & 0 & 0 & 0 & q 1_{b_{\gamma}}  & 0 & 0 & 0 & 0 & 0 & 1_{b_{\gamma}} \\
\hline
0 & 1_{a_3} & 0 & 0 & 0 & 0 & 0 & 0 & 0 & 0 & 0 & 0 \\
0 & 0 & 0 & 0 & 0 & 0 & 0 & 0 & 1_{a_6} & 0 & 0 & 0 \\
0 & 0 & 0 & 0 & 1_{b_{\beta}} & 0 & 0 & 0 & 0 & 0 & 1_{b_{\beta}} & 0 \\
0 & 0 & 0 & 0 & 0 & 1_{b_{\gamma}} & 0 & 0 & 0 & 0 & 0 & 1_{b_{\beta}} 
\end{array} 
\]
We will now determine the structure of the base-change matrices $B$ of isoclasses of $C_2 \ast C_3$-representations $M$ in a Zariski open neighborhood of $[M_0]$ in $\wis{iss}_{\sigma}~C_2 \ast C_3$.

As $M_0$ is semi-simple, its isomorphism class forms a Zariski closed orbit $\Oscr(M_0)$ in the smooth irreducible component $\wis{rep}_{\sigma}~C_2 \ast C_3$ under the action of $GL(\sigma) = GL_a \times GL_b \times GL_x \times GL_y \times GL_z$. The stabilizer subgroup $Stab(M_0)$ is the automorphism group and is the subgroup of $GL(\sigma)$ we will denote by $GL(\tau) = GL_{a_1} \times GL_{a_2} \times GL_{a_3} \times GL_{a_4} \times GL_{a_5} \times GL_{a_6} \times GL_{b_{\alpha}} \times GL_{b_{\beta}} \times GL_{b_{\gamma}}$. The normal space to the orbit $\Oscr(M_0)$ can be identified as $GL(\tau)$-representation with the vectorspace of self-extensions $Ext^1_{C_2 \ast C_3}(M_0,M_0)$, see for example \cite[II.2.7]{Kraft}. The Luna slice theorem, see for example \cite[\S 4.2]{LBBook}, asserts that the action map
\[
GL(\sigma) \times^{GL(\tau)} Ext^1_{C_2 \ast C_3}(M_0,M_0) \rTo \wis{rep}_{\sigma}~C_2 \ast C_3 \]
sending the class of $(g,\vec{n})$ in the associated fibre bundle to the $C_2 \ast C_3$-representation $g.(M+\vec{n})$ is a $GL(\sigma)$-equivariant \'etale map with a Zariski dense image. Taking $GL(\sigma)$-quotients on both sides we obtain an \'etale map
\[
Ext^1_{C_2 \ast C_3}(M_0,M_0)/GL(\tau) \rTo \wis{iss}_{\sigma}~C_2 \ast C_3 \]
with a Zariski dense image. The crucial observation to make is that it follows from the theory of local quivers, \cite[\S 4.2]{LBBook}, that as a $GL(\tau)$-representation $Ext^1_{C_2 \ast C_3}(M_0,M_0)$ is isomorphic to $\wis{rep}_{\tau}~Q$ for the quiver $Q$ having $9$ vertices (one for each of the distinct simple factors of $M_0$) and having as many directed arrows from the vertex corresponding to the simple factor $S$ to that of the simple factor $T$ as is the dimension of the space $Ext^1_{C_2 \ast C_3}(S,T)$. This then allows to identify the quotient variety $Ext^1_{C_2 \ast C_3}(M_0,M_0)/GL(\tau)$ with the affine variety $\wis{iss}_{\tau}~Q$ whose points are the isoclasses of semi-simple representations of $Q$ of dimension-vector $\tau=(a_1,a_2,a_3,a_4,a_5,a_6,b_{\alpha},b_{\beta},b_{\gamma})$, and the action map induces an \'etale map with dense image
\[
\wis{iss}_{\tau}~Q \rTo \wis{iss}_{\sigma}~C_2 \ast C_3 \]

Computing the normal space to the orbit $\Oscr(M_0)$ as in the proof of \cite[Thm. 4]{LBbraid1} but for the more complicated representation $M_0$ one obtains that the sub quiver of $Q$ on the $6$ vertices corresponding to the $1$-dimensional simple components $S_1,\hdots,S_6$ coincides with that of \cite{LBbraid1}, that is corresponds to the quiver-setting
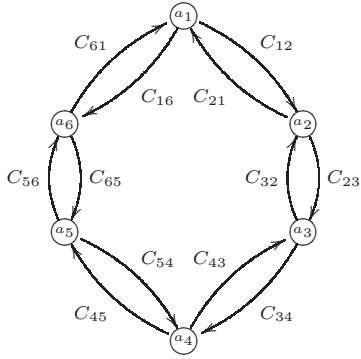
\[
\xymatrix@=1.1cm{
& \vtx{a_1} \ar@/^/[ld]^{C_{16}} \ar@/^/[rd]^{C_{12}} & \\
\vtx{a_6} \ar@/^/[ru]^{C_{61}}  \ar@/^/[d]^{C_{65}} & & \vtx{a_2} \ar@/^/[lu]^{C_{21}} \ar@/^/[d]^{C_{23}} \\
\vtx{a_5} \ar@/^/[u]^{C_{56}}  \ar@/^/[rd]^{C_{54}} & & \vtx{a_3} \ar@/^/[u]^{C_{32}} \ar@/^/[ld]^{C_{34}} \\
& \vtx{a_4} \ar@/^/[lu]^{C_{45}} \ar@/^/[ru]^{C_{43}}  & }
\]
The additional quiver-setting depending on the $3$ vertices corresponding to the $2$-dimensional simple factors $T_1(q)$, $T_2(q)$ and $T_3(q)$ can be verified to be
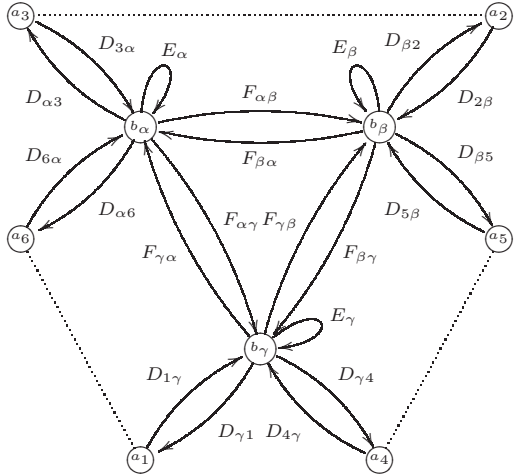
\[
\xymatrix@=1.1cm{
\vtx{a_3} \ar@/^/[rd]^{D_{3 \alpha}} \ar@{.}[rrrr] & & & & \vtx{a_2} \ar@/^/[ld]^{D_{2 \beta}} \\
& \vtx{b_{\alpha}} \ar@/^/[lu]^{D_{\alpha 3}} \ar@/^/[rr]^{F_{\alpha \beta}} \ar@(u,ur)^{E_{\alpha}} \ar@/^/[ld]^{D_{\alpha 6}} \ar@/^/[rdd]^{F_{\alpha \gamma}} & & \vtx{b_{\beta}} \ar@/^/[ll]^{F_{\beta \alpha}} \ar@/^/[ru]^{D_{\beta 2}} \ar@(u,ul)_{E_{\beta}} \ar@/^/[rd]^{D_{\beta 5}} \ar@/^/[ldd]^{F_{\beta \gamma}} & \\
\vtx{a_6} \ar@/^/[ru]^{D_{6 \alpha}} \ar@{.}[rdd] & & & & \vtx{a_5} \ar@/^/[lu]^{D_{5 \beta}} \ar@{.}[ldd]  \\
& & \vtx{b_{\gamma}} \ar@/^/[luu]^{F_{\gamma \alpha}} \ar@/^/[ruu]^{F_{\gamma \beta}} \ar@/^/[ld]^{D_{\gamma 1}} \ar@/^/[rd]^{D_{\gamma 4}} \ar@(ur,r)^{E_{\gamma}} & & \\
& \vtx{a_1} \ar@/^/[ru]^{D_{1 \gamma}} & & \vtx{a_4}  \ar@/^/[lu]^{D_{4 \gamma}} &}
\]
which concludes the proof of the following:

\begin{theorem}
The \'etale action map $GL(\sigma) \times^{GL(\tau)} \wis{rep}_{\tau}~Q \rTo \wis{rep}_{\sigma}~C_2 \ast C_3$ sends a $\tau$-dimensional $Q$-representation to the $C_2 \ast C_3$-representation determined by the base-change matrix $B$
{\tiny
\[
\begin{array}{|cccccc|cccccc|}
1_{a_1} & 0 & 0 & 0 & 0 & 0 & C_{21} & 0 & C_{61} & 0 & 0 & D_{\gamma 1} \\
0 & C_{34} & C_{54} & 0 & 0 & D_{\gamma 4} & 0 & 1_{a_4} & 0 & 0 & 0 & 0 \\
0 & D_{3 \alpha} & 0 & q 1_{b_{\alpha}} + E_{\alpha} & 0 & 0 & 0 & 0 & D_{6 \alpha} & 1_{b_{\alpha}} & 0 & F_{\gamma \alpha} \\
0 & 0 & 0 & 0 & q 1_{b_{\beta}} & F_{\gamma \beta} & 0 & 0 & 0 & 0 & 1_{b_{\beta}} & 0  \\
\hline
C_{12} & C_{32} & 0 & 0 & D_{\beta 2} & 0 & 1_{a_2} & 0 & 0 & 0 & 0 & 0 \\
0 & 0 & 1_{a_5} & 0 & 0 & 0 & 0 & C_{45} & C_{65} & 0 & D_{\beta 5} & 0 \\
0 & 0 & 0 & 1_{b_{\alpha}} & 0 & 0 & 0 & 0 & 0 & 1_{b_{\alpha}} & F_{\beta \alpha} & 0 \\
D_{1 \gamma} & 0 & 0 & 0 & F_{\beta \gamma} & q 1_{b_{\gamma}} + E_{\gamma} & 0 & D_{4 \gamma} & 0 & 0 & 0 & 1_{b_{\gamma}} \\
\hline
0 & 1_{a_3} & 0 & 0 & 0 & 0 & C_{23} & C_{43} & 0 & D_{\alpha 3} & 0 & 0 \\
C_{16} & 0 & C_{56} & D_{\alpha 6} & 0 & 0 & 0 & 0 & 1_{a_6} & 0 & 0 & 0 \\
0 & 0 & D_{5 \beta} & 0 & 1_{b_{\beta}}+E_{\beta} & 0 & D_{2 \beta} & 0 & 0 & F_{\alpha \beta} & 1_{b_{\beta}} & 0 \\
0 & 0 & 0 & F_{\alpha \gamma} & 0 & 1_{b_{\gamma}} & 0 & 0 & 0 & 0 & 0 & 1_{b_{\gamma}} 
\end{array} 
\]
}

\noindent
Under this map, simple $Q$-representations are mapped to irreducible $C_2 \ast C_3$-representations, and if the coefficients of the block-matrices $C_{ij},D_{ij},E_i$ and $F_{ij}$ occurring in $B$ give a parametrization of a Zariski open subset of the quotient variety $\wis{iss}_{\tau}~Q$, then the $n$-dimensional representations of the $3$-string braid group $B_3$ given by
\[
\begin{cases}
\sigma_1 \mapsto  \lambda^{1/6} B^{-1} \begin{bmatrix} 1_{x} & 0 & 0 \\ 0 & \rho^2 1_{y} & 0 \\ 0 & 0 & \rho 1_{z} \end{bmatrix} B \begin{bmatrix} 1_{a} & 0 \\ 0 & -1_{b} \end{bmatrix} \\
\sigma_2 \mapsto  \lambda^{1/6} \begin{bmatrix} 1_{a} & 0 \\ 0 & -1_{b} \end{bmatrix} B^{-1}  \begin{bmatrix} 1_{x} & 0 & 0 \\ 0 & \rho^2 1_{y} & 0 \\ 0 & 0 & \rho 1_{z} \end{bmatrix} B
\end{cases}
\begin{cases}
a=a_1+a_3+a_5+b_{\alpha} + b_{\beta} \\
b=a_2+a_4+a_6+b_{\alpha}+b_{\gamma} \\
x=a_1+a_4+b_{\alpha}+b_{\beta} \\
y=a_2+a_5+b_{\alpha}+b_{\gamma} \\
z=a_3+a_6+b_{\beta}+b_{\gamma}
\end{cases}
\]
contain a Zariski dense set of irreducible $B_3$-representations in the component $X_{\sigma}$ of $\wis{iss}_n~B_3$.
\end{theorem}

\section{The main result}

In view of the previous section it remains to find for each $\sigma=(a,b;x,y,z)$ satisfying
\[
a+b=n=x+y+z \quad \text{and} \quad x=max(x,y,z) \leq b=min(a,b) \]
a judiciously chosen dimension-vector $\tau=(a_1,a_2,a_3,a_4,a_5,a_6,b_{\alpha},b_{\beta},b_{\gamma})$ of type $\sigma$ together with an explicit rational parametrization of $\wis{iss}_{\tau}~Q$. We will separate this investigation in two cases, sharing the same underlying strategy. 
First we choose $a_1,a_2,a_3,a_4,a_5,a_6$ such that $\sigma_1 = (a_1+a_3+a_5,a_2+a_4+a_6;a_1+a_4,a_2+a_5,a_3+a_6)$ is a component containing simples and such that we have an explicit rational parametrization of the isoclasses of the quiver-setting
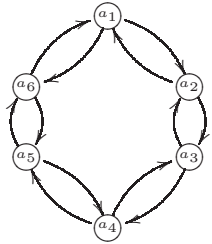
\[
\xymatrix@=0.6cm{
& \vtx{a_1} \ar@/^/[ld] \ar@/^/[rd] & \\
\vtx{a_6} \ar@/^/[ru]  \ar@/^/[d] & & \vtx{a_2} \ar@/^/[lu] \ar@/^/[d] \\
\vtx{a_5} \ar@/^/[u]  \ar@/^/[rd] & & \vtx{a_3} \ar@/^/[u] \ar@/^/[ld] \\
& \vtx{a_4} \ar@/^/[lu] \ar@/^/[ru]  & }
\]
The upshot being that for a general representation the stabilizer subgroup reduces to $\C^* (1_{a_1} \times \hdots \times 1_{a_6})$. But then, the additional arrows $D_{ij}$ and $E_{i}$, that is the quiver setting
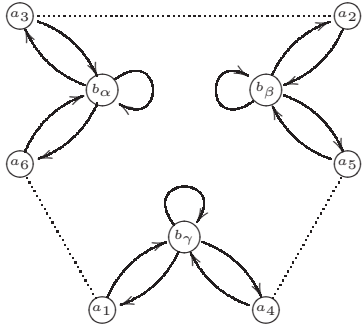
\[
\xymatrix@=0.6cm{
\vtx{a_3} \ar@/^/[rd] \ar@{.}[rrrr] & & & & \vtx{a_2} \ar@/^/[ld] \\
& \vtx{b_{\alpha}} \ar@/^/[lu]  \ar@(ur,dr) \ar@/^/[ld]  & & \vtx{b_{\beta}}  \ar@/^/[ru]  \ar@(dl,ul) \ar@/^/[rd]  & \\
\vtx{a_6} \ar@/^/[ru] \ar@{.}[rdd] & & & & \vtx{a_5} \ar@/^/[lu] \ar@{.}[ldd]  \\
& & \vtx{b_{\gamma}}   \ar@/^/[ld] \ar@/^/[rd] \ar@(ul,ur) & & \\
& \vtx{a_1} \ar@/^/[ru] & & \vtx{a_4}  \ar@/^/[lu]  &}
\]
give three settings corresponding to quiver settings of canonical linear systems with $m=p=a_i+a_{i+3}$ and the results of section~2 give a rational parametrization of the isoclasses and further reduces the stabilizer subgroup to $\C^*(1_{a_1} \times \hdots \times 1_{a_6} \times 1_{b_{\alpha}} \times 1_{b_{\beta}} \times 1_{b_{\gamma}})$. This then leaves the trivial action on the remaining arrows $F_{ij}$ and hence these generic matrices conclude the desired rational parametrization.

\subsection{Case 1 : $a > b$} Define $d=a-b$, $e=d-1$, $f=b-z$, $g=b-y$ and $h=b-x$, then the dimension-vector
\[
\tau=(a_1,a_2,a_3,a_4,a_5,a_6,b_{\alpha},b_{\beta},b_{\gamma}) = (d,e,e,0,1,1,f,g,h) \]
is of type $\sigma$. If we denote by 
\[
\begin{cases}
\ast &\quad~\text{ a generic matrix} \\
| &\quad~\text{the column vector}~ \begin{bmatrix} 1\\ 0 \\ \vdots\\ 0 \end{bmatrix} \\
\overline{1}_n &\quad~\text{the $n+1 \times n$ matrix}~\begin{bmatrix} \underline{0} \\ 1_n \end{bmatrix}
\end{cases}
\]
 and the (reduced) companion matrices as in lemma~\ref{lemma2}, then using lemma~\ref{case1} a rational parametrization of the first stage is given by the representations
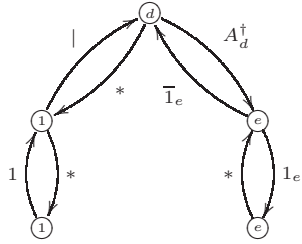
\[
\xymatrix@=1.1cm{
& \vtx{d} \ar@/^/[ld]^{\ast} \ar@/^/[rd]^{A_d^{\dagger}} & \\
\vtx{1} \ar@/^/[ru]^{|}  \ar@/^/[d]^{\ast} & & \vtx{e} \ar@/^/[lu]^{\overline{1}_e} \ar@/^/[d]^{1_e} \\
\vtx{1} \ar@/^/[u]^{1}   & & \vtx{e} \ar@/^/[u]^{\ast}  \\
& \vtx{0}   & }
\]
By lemma~\ref{lemma1} a rational parametrization of the second stage is then given by the representations
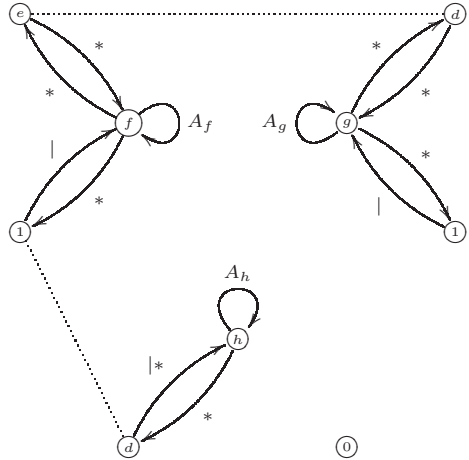
\[
\xymatrix@=1.1cm{
\vtx{e} \ar@/^/[rd]^{\ast} \ar@{.}[rrrr] & & & & \vtx{d} \ar@/^/[ld]^{\ast} \\
& \vtx{f} \ar@/^/[lu]^{\ast}  \ar@(ur,dr)^{A_f} \ar@/^/[ld]^{\ast}  & & \vtx{g}  \ar@/^/[ru]^{\ast}  \ar@(dl,ul)^{A_g} \ar@/^/[rd]^{\ast}  & \\
\vtx{1} \ar@/^/[ru]^{|} \ar@{.}[rdd] & & & & \vtx{1} \ar@/^/[lu]^{|}   \\
& & \vtx{h}   \ar@/^/[ld]^{\ast}  \ar@(ul,ur)^{A_h} & & \\
& \vtx{d} \ar@/^/[ru]^{| \ast} & & \vtx{0}   &}
\]
This concludes the proof of

\begin{theorem} A Zariski dense rational parametrization of the component $X_{\sigma}$ of $\wis{iss}_n B_3$ where $\sigma=(a,b;x,y,z)$ with $a > b$ is given by the representations
\[
\begin{cases}
\sigma_1 \mapsto  \lambda^{1/6} B^{-1} \begin{bmatrix} 1_{x} & 0 & 0 \\ 0 & \rho^2 1_{y} & 0 \\ 0 & 0 & \rho 1_{z} \end{bmatrix} B \begin{bmatrix} 1_{a} & 0 \\ 0 & -1_{b} \end{bmatrix} \\
\sigma_2 \mapsto  \lambda^{1/6} \begin{bmatrix} 1_{a} & 0 \\ 0 & -1_{b} \end{bmatrix} B^{-1}  \begin{bmatrix} 1_{x} & 0 & 0 \\ 0 & \rho^2 1_{y} & 0 \\ 0 & 0 & \rho 1_{z} \end{bmatrix} B
\end{cases}
\]
for all $n \times n$ matrices $B$ of the form
\[
\begin{array}{|cccccc|ccccc|}
1_d & 0 & 0 & 0 & 0 & 0 & \overline{1}_e &  | & 0 & 0 & \ast \\
0 & \ast & 0 & q 1_f +A_f & 0 & 0 & 0 & | & 1_f  & 0 & \ast \\
0 & 0 & 0 & 0 & q 1_g & \ast & 0 & 0 & 0 & 1_g & 0 \\
\hline
A_d^{\dagger} & \ast & 0 & 0 & \ast & 0 & 1_e & 0 & 0 & 0 & 0 \\
0 & 0 & 1 & 0 & 0 & 0 & 0 & \ast & 0 & \ast & 0 \\
0 & 0 & 0 & 1_f & 0 & 0 & 0 & 0 & 1_f & \ast & 0 \\
| \ast & 0 & 0 & 0 & \ast & q 1_h + A_h & 0 & 0 & 0 & 0 & 1_h \\
\hline
0 & 1_e & 0 & 0 & 0 & 0 & 1_e & 0 & \ast & 0 & 0 \\
\ast & 0 & 1 & \ast & 0 & 0 & 0 & 1 & 0 & 0 & 0 \\
0 & 0 & | & 0 & 1_g + A_g & 0 & \ast & 0 & \ast & 1_g & 0 \\
0 & 0 & 0 & \ast & 0 & 1_h & 0 & 0 & 0 & 0 & 1_g
\end{array}
\]
where $d=a-b$, $e=d-1$, $f=b-z$, $g=b-y$ and $h=b-x$.
\end{theorem}

\subsection{Case 2 : $a=b$}  Define $c=x+y+1-a$, $g=a-y-1$ and $h=a-x$, which corresponds to the decomposition
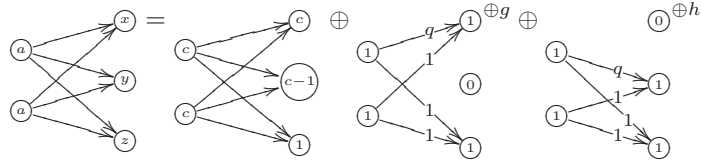
\[
\xymatrix@=.1cm{
& & & & \vtx{x} \\
\vtx{a} \ar[rrrru] \ar[rrrrd] \ar[rrrrddd] & & & & \\
& & & & \vtx{y} \\
\vtx{a} \ar[rrrruuu] \ar[rrrru] \ar[rrrrd] & & & & \\
& & & & \vtx{z}} = 
\xymatrix@=.1cm{
& & & & \vtx{c} \\
\vtx{c} \ar[rrrru] \ar[rrrrd] \ar[rrrrddd] & & & & \\
& & & & \vtx{c-1} \\
\vtx{c} \ar[rrrruuu] \ar[rrrru] \ar[rrrrd] & & & & \\
& & & & \vtx{1}} \oplus
\xymatrix@=.1cm{
& & & & \vtx{1} \\
\vtx{1} \ar[rrrru]|(0.6){q}  \ar[rrrrddd]|(0.6){1} & & & & \\
& & & & \vtx{0} \\
\vtx{1}  \ar[rrrruuu]|(0.6){1} \ar[rrrrd]|(0.6){1} & & & & \\
& & & &  \vtx{1} }^{\oplus g} \oplus
\xymatrix@=.1cm{
& & & & \vtx{0} \\
\vtx{1}  \ar[rrrrd]|(0.6){q} \ar[rrrrddd]|(0.7){1} & & & & \\
& & & & \vtx{1} \\
\vtx{1} \ar[rrrru]|(0.6){1}  \ar[rrrrd]|(0.6){1} & & & & \\
& & & &  \vtx{1} }^{\oplus h}
\]
If $c$ is odd, define $c=2d+1$, $e=d+1$ and $f=d-1$, then the dimension vector
\[
\tau=(a_1,a_2,a_3,a_4,a_5,a_6,b_{\alpha},b_{\beta},b_{\gamma}) = (e,e,1,d,f,0,0,g,h) \]
is of type $\sigma$. Then, using lemma~\ref{lemma2} a rational parametrization for the first stage is given by the representations
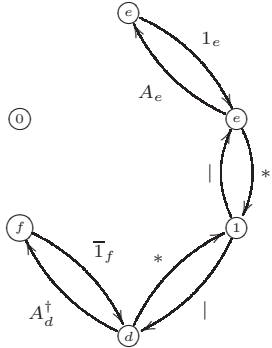
\[
\xymatrix@=1.1cm{
& \vtx{e}  \ar@/^/[rd]^{1_e} & \\
\vtx{0}  & & \vtx{e} \ar@/^/[lu]^{A_e} \ar@/^/[d]^{\ast} \\
\vtx{f} \ar@/^/[rd]^{\overline{1}_f}   & & \vtx{1} \ar@/^/[u]^{|} \ar@/^/[ld]^{|} \\
& \vtx{d} \ar@/^/[lu]^{A_d^{\dagger}} \ar@/^/[ru]^{\ast}  & }
\]
Using lemma~\ref{lemma1} we then get that a rational parametrization of the second stage is given by the following representations
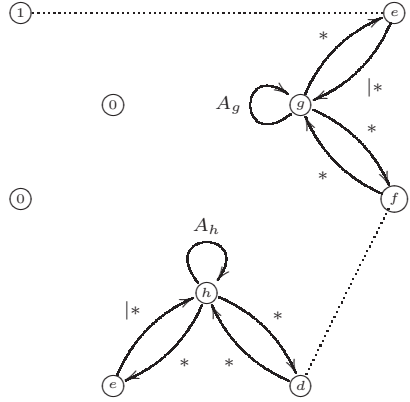
\[
\xymatrix@=0.9cm{
\vtx{1} \ar@{.}[rrrr] & & & & \vtx{e} \ar@/^/[ld]^{ | \ast} \\
& \vtx{0}   & & \vtx{g}  \ar@/^/[ru]^{\ast}  \ar@(dl,ul)^{A_g} \ar@/^/[rd]^{\ast}  & \\
\vtx{0}  & & & & \vtx{f} \ar@/^/[lu]^{\ast}  \ar@{.}[ldd]  \\
& & \vtx{h}   \ar@/^/[ld]^{\ast}  \ar@(ul,ur)^{A_h} \ar@/^/[rd]^{\ast} & & \\
& \vtx{e} \ar@/^/[ru]^{| \ast} & & \vtx{d} \ar@/^/[lu]^{\ast}  &}
\]
If $c$ is even, we can define $c=2e$ and $f=e-1$ in which case the dimension vector
\[
\tau=(a_1,a_2,a_3,a_4,a_5,a_6,b_{\alpha},b_{\beta},b_{\gamma}) = (e,e,1,e,f,0,0,g,h) \]
is of type $\sigma$ and exactly the same representations give a rational parametrization of both stages if we replace all occurrences of $d$ by $e$. This then concludes the proof of

\begin{theorem} A Zariski dense rational parametrization of the component $X_{\sigma}$ of $\wis{iss}_n B_3$ where $\sigma=(a,b;x,y,z)$ with $a = b$ is given by the representations
\[
\begin{cases}
\sigma_1 \mapsto  \lambda^{1/6} B^{-1} \begin{bmatrix} 1_{x} & 0 & 0 \\ 0 & \rho^2 1_{y} & 0 \\ 0 & 0 & \rho 1_{z} \end{bmatrix} B \begin{bmatrix} 1_{a} & 0 \\ 0 & -1_{b} \end{bmatrix} \\
\sigma_2 \mapsto  \lambda^{1/6} \begin{bmatrix} 1_{a} & 0 \\ 0 & -1_{b} \end{bmatrix} B^{-1}  \begin{bmatrix} 1_{x} & 0 & 0 \\ 0 & \rho^2 1_{y} & 0 \\ 0 & 0 & \rho 1_{z} \end{bmatrix} B
\end{cases}
\]
for all $n \times n$ matrices $B$ of the form
\[
\begin{array}{|ccccc|cccc|}
1_{e} & 0 & 0 &  0 & 0 & A_e & 0 &  0 & \ast \\
0 & | & \overline{1}_f &  0 & \ast & 0 & 1_{d} &  0 & 0 \\
0 & 0 & 0 &  q 1_{g} & \ast & 0 & 0 &  1_{g} & 0  \\
\hline
1_e & | & 0 &  \ast & 0 & 1_{e} & 0 &  0 & 0 \\
0 & 0 & 1_{f} &  0 & 0 & 0 & A_d^{\dagger} &  \ast & 0 \\
| \ast & 0 & 0 &  \ast & q 1_{h} + A_h & 0 & \ast &  0 & 1_{h} \\
\hline
0 & 1 & 0 &  0 & 0 & \ast & \ast &  0 & 0 \\
0 & 0 & \ast &  1_{g}+A_g & 0 & | \ast & 0 &   1_{g} & 0 \\
0 & 0 & 0 &  0 & 1_{h} & 0 & 0 &  0 & 1_{h} 
\end{array} 
\]
where $g=a-y-1$, $h=a-x$ and if $c=x+y+1-a$ is odd we take $c=2d+1$, $e=d+1$ and $f=d-1$ whereas if $c=x+y+1-a$ is even we take $c=2e$ and $f=e-1$ and we replace all occurrences of $d$ in the matrix to $e$.
\end{theorem}

\end{document}